\documentclass{daj}

\usepackage{amssymb,mathtools,amsthm,amsmath}

\usepackage{enumitem}
\usepackage{braket}

\theoremstyle{plain}
\newtheorem{theorem}{\bf Theorem}

\newtheorem{proposition}[theorem]{\bf Proposition}
\newtheorem{corollary}[theorem]{\bf Corollary}
\newtheorem{lemma}[theorem]{\bf Lemma}
\theoremstyle{definition}

\newenvironment{remark}[1][Remark.]{\begin{trivlist}
		\item[\hskip \labelsep {\bfseries #1}]}{\end{trivlist}}

\newenvironment{example}[1][Example.]{\begin{trivlist}
		\item[\hskip \labelsep {\bfseries #1}]}{\end{trivlist}}

\numberwithin{theorem}{section}
\numberwithin{equation}{section}

\newcommand{\FF}{{\mathbb F}}
\newcommand{\Rea}{{\mathbb R}}
\DeclareMathOperator{\Ker}{Ker}
\DeclareMathOperator{\Ima}{Im}
\newcommand{\lfrac}[2]{\left\lfloor \frac{#1}{#2}\right\rfloor}
\newcommand{\sumwedge}[2]{#2^{[#1]}}
\newcommand{\maxksum}[2]{\sum_{j=1}^{#1}\eiglargest{j}{#2}}

\newcommand{\ithmaxksum}[3]{S_{#1,#2}^{\downarrow}(#3)}
\newcommand{\ithminksum}[3]{S_{#1,#2}^{\uparrow}(#3)}
\newcommand{\lap}[2]{L_{#1}(#2)}
\newcommand{\lapw}[3]{L_{#1}^{#3}(#2)}
\newcommand{\glapw}[2]{L^{#2}(#1)}
\newcommand{\simlapw}[2]{\mathcal{L}^{#2}(#1)}

\newcommand{\eigsmallest}[2]{\lambda_{#1}^{\uparrow}(#2)}
\newcommand{\eiglargest}[2]
{\lambda_{#1}^{\downarrow}(#2)}

\dajAUTHORdetails{%
  title = {Laplacian eigenvalues of independence complexes via additive compound matrices}, %% please capitalize all significant words
  author = {Alan Lew},
  plaintextauthor = {Alan Lew},
  plaintexttitle = {Laplacian eigenvalues of independence complexes via additive compound matrices}, %%  title without math or LaTeX
  keywords = {high dimensional Laplacian, additive compound matrix, independence complex},
}   %%% END \dajAUTHORdetails

\dajEDITORdetails{%
   year={2024},
   number={15},
   received={28 July 2023},   % received date: example: 7 January 2017
   published={19 December 2024},  % published date
   doi={10.19086/da.125857},       % XXX = number of paper, e.g. da006 for paper#6
}   %%% END \dajEDITORdetails

\begin{document}

\begin{frontmatter}[classification=text]
%% EDITOR: this will force the keywords to appear right after the Abstract.
%%   If the abstract is too long and would force the keywords off the
%%   front page, please comment out % [classification=text] above
%%   This way the keywords will be floated on the bottom of the first page
%%   even though the Abstract spills over to the next page.

%%% AUTHOR: Title goes here.  This line is optional.  You must use it
%%   if title has footnote attached or requires nontrivial typesetting,
%%   e.g., inclusion of linebreaks to force nice layout.
\title{Laplacian eigenvalues of independence complexes via additive compound matrices} %% please capitalize all significant words

%%% AUTHOR:
%%% List all authors. If you wish, place grant acknowledgements in \thanks.
%%% In brackets include a short tag for each author.
\author[pgom]{Alan Lew}%\thanks{Supported by...}}

%%% AUTHOR: Abstract goes here
\begin{abstract}
  The independence complex of a graph $G=(V,E)$ is the simplicial complex $I(G)$ on vertex set $V$ whose simplices are the independent sets in $G$. We present new lower bounds on the eigenvalues of the $k$-dimensional Laplacian   $L_k(I(G))$ in terms of the eigenvalues of the graph Laplacian $L(G)$.
  As a consequence, we show that for all $k\geq 0$, the dimension of the $k$-th reduced homology group (with real coefficients) of $I(G)$  is at most
 \[  \left| \left\{ 1\leq i_1<\cdots<i_{k+1}\leq |V| : \, \lambda_{i_1}+\lambda_{i_2}+\cdots+\lambda_{i_{k+1}} \geq |V|\right\}\right|,
  \]
where $\lambda_1\geq\lambda_2\geq  \cdots\geq \lambda_{|V|}=0$ are the eigenvalues of $L(G)$. In particular, if $k$ is the minimal number such that the sum of the $k$ largest eigenvalues of $L(G)$ is at least $|V|$, then $\tilde{H}_i(I(G);\mathbb{R})=0$ for all $i\leq k-2$.
   This extends previous results by Aharoni, Berger and Meshulam. 
   Our proof relies on a relation between the $k$-dimensional Laplacian $L_k(I(G))$ and the $(k+1)$-th additive compound matrix of $L_0(I(G))$, which is an $\binom{n}{k+1}\times\binom{n}{k+1}$ matrix whose eigenvalues are all the possible sums of $k+1$ eigenvalues of the $0$-dimensional Laplacian. Our results apply  also in the more general setting of vertex-weighted Laplacian matrices. 
\end{abstract}
\end{frontmatter}

%%% AUTHOR: body of paper starts here

\section{Introduction}

Let $G=(V,E)$ be a graph, and let $w:V\to\Rea_{\geq 0}$ be a weight function on $V$. We say that $w$ is \emph{positive} if $w(v)>0$ for all $v\in V$. The \emph{vertex-weighted Laplacian} on $G$  is the matrix $L^w(G)\in \Rea^{V\times V}$ defined by
\begin{equation}\label{eq:glap_def}
    \glapw{G}{w}_{u,v}=\begin{dcases}
    \sum_{u'\in N_G(u)} w(u') & \text{if } u=v,\\
    -w(v) & \text{if } \{u,v\}\in E,\\
    0 & \text{ otherwise,}
    \end{dcases}
\end{equation}
for all $u,v\in V$, where $N_G(u)$ denotes the set of neighbors of $u$ in $G$.
In the special case where $w(v)=1$ for all $v\in V$, we obtain $L^w(G)=L(G)$, the combinatorial Laplacian matrix on $G$. Note that $L^{w}(G)$ is not in general a symmetric matrix. However, assuming that $w$ is positive,  
$\glapw{G}{w}$ is similar to the symmetric matrix $\mathcal{L}^w(G)\in \Rea^{V\times V}$ defined by
\begin{equation}\label{eq:simlap}
    \mathcal{L}^w(G)_{u,v}=\begin{dcases}
    \sum_{u'\in N_G(u)} w(u') & \text{if } u=v,\\
    -\sqrt{w(u) w(v)} & \text{if } \{u,v\}\in E,\\
    0 & \text{ otherwise,}
    \end{dcases}
\end{equation}
for all $u,v\in V$. Indeed, we can write $\glapw{G}{w}=W^{-1/2}\simlapw{G}{w}  W^{1/2}$, where $W$ is the diagonal matrix with elements $W_{u,u}= w(u)$ for all $u\in V$ (see \cite{chung1996combinatorial}). If $w$ is non-negative, then we can write $\glapw{G}{w}=\lim_{\epsilon\to 0} \glapw{G}{w_{\epsilon}}$ and $\simlapw{G}{w}=\lim_{\epsilon\to 0} \simlapw{G}{w_{\epsilon}}$, where for all $\epsilon>0$ we define $w_{\epsilon}:V\to \Rea_{>0}$ by $w_{\epsilon}(v)=w(v)$ if $w(v)>0$ and $w_{\epsilon}(v)=\epsilon$ if $w(v)=0$, for all $v\in V$. Therefore, since $\glapw{G}{w_{\epsilon}}$ is similar to $\simlapw{G}{w_{\epsilon}}$ for all $\epsilon>0$, and by the continuity of eigenvalues, $\glapw{G}{w}$ and $\simlapw{G}{w}$ have the same eigenvalues also in this case.

The \emph{independence complex} of $G$ is the simplicial complex $I(G)$ on vertex set $V$ whose simplices are the independent sets in $G$. For $k\geq -1$, let $f_k(I(G))$ be the number of $k$-dimensional simplices in $I(G)$, let $C^k(I(G);\Rea)$ be the space of real valued $k$-cochains on $I(G)$, and let $d_k: C^k(I(G);\Rea)\to C^{k+1}(I(G);\Rea)$ be the $k$-th coboundary operator. A positive weight function $w:V\to \Rea_{>0}$ induces an inner product on $C^k(I(G);\Rea)$ (see Section \ref{sec:higher_laplacians}); let $d_k^*: C^{k+1}(I(G);\Rea)\to C^k(I(G);\Rea)$ be the adjoint of $d_k$ with respect to these inner products. The \emph{$w$-weighted $k$-dimensional Laplacian} on $I(G)$ is the linear operator \[\lapw{k}{I(G)}{w}= d_k^*d_k+d_{k-1}d_{k-1}^*.\] In the special case where $w(v)=1$ for all $v\in V$, we write $\lapw{k}{I(G)}{w}=\lap{k}{I(G)}$.
The simplicial Hodge theorem states that for every $w:V\to\Rea_{>0}$, $\Ker(\lapw{k}{I(G)}{w})$ is isomorphic to the $k$-th homology group $\tilde{H}_k(I(G);\Rea)$. The definition of $\lapw{k}{I(G)}{w}$ can be extended to the case of non-negative weight functions $w$, in which case we have $\dim(\tilde{H}_k(I(G);\Rea))\leq \dim(\Ker(\lapw{k}{I(G)}{w}))$ (see Lemma \ref{lemma:homology_from_eigenvalues_degenerate_case}).

For a matrix $M\in \Rea^{n\times n}$ with real eigenvalues and $1\leq i\leq n$, we denote by $\eigsmallest{i}{M}$ the $i$-th smallest eigenvalue of $M$ and by $\eiglargest{i}{M}$ its $i$-th largest eigenvalue (so that $\eiglargest{i}{M}=\eigsmallest{n+1-i}{M}$). For any $k$, let
\[
\mathcal{S}_k(M)=\left\{ \sum_{i\in I} \eigsmallest{i}{M} :\, I\in \binom{[n]}{k}\right\}
\]
be the multiset consisting of all possible sums of $k$ eigenvalues of $M$. Let $\ithminksum{k}{i}{M}$ be the $i$-th smallest element of $\mathcal{S}_k(M)$, and $\ithmaxksum{k}{i}{M}$ be its $i$-th largest element.

In \cite{aharoni2005eigenvalues}, Aharoni, Berger and Meshulam studied the relation between the smallest eigenvalues of successive high dimensional Laplacians of $I(G)$.
As a principal consequence, they obtained the following result relating the minimal eigenvalue of $\lap{k}{I(G)}$ to the maximal eigenvalue of $L(G)$.

\begin{theorem}[Aharoni, Berger, Meshulam {\cite{aharoni2005eigenvalues}}]
\label{thm:ABM_main}
Let $G=(V,E)$ be a graph. Then
\[
\eigsmallest{1}{\lap{k}{I(G)}}\geq |V|-(k+1)\eiglargest{1}{L(G)}.
\]
\end{theorem}
The \emph{homological connectivity} of $I(G)$, denoted by $\eta(I(G))$, is defined as the maximal $k$ such that $\tilde{H}_i(I(G);\Rea)=0$ for all $i\leq k-2$. The following result is an immediate consequence of Theorem \ref{thm:ABM_main}.

\begin{theorem}[Aharoni, Berger, Meshulam {\cite[Thm. 4.1]{aharoni2005eigenvalues}}]
\label{thm:ABM_connectivity}
Let $G=(V,E)$ be a graph. Then
\[
\eta(I(G))\geq \frac{|V|}{\eiglargest{1}{L(G)}}.
\]
\end{theorem}

Here, we continue the study of the spectra of high dimensional Laplacian operators on independence complexes. Our main result is the following extension of Theorem \ref{thm:ABM_main}.

\begin{theorem}  \label{thm:main_independence}
Let $G=(V,E)$ be a graph, and let $w:V\to \Rea_{\geq 0}$. Then, for all $k\geq 0$ and $1\leq i\leq f_k(I(G))$, 
\[
    \eigsmallest{i}{\lapw{k}{I(G)}{w}}\geq \left(\sum_{v\in V} w(v)\right)-\ithmaxksum{k+1}{i}{L^{w}(G)}.
\]
\end{theorem}
Theorem \ref{thm:ABM_main} follows from the $i=1$ case of Theorem \ref{thm:main_independence}, using the fact that $\ithmaxksum{k+1}{1}{L(G)}\leq (k+1)\eiglargest{1}{L(G)}$ for all $k\geq 0$.  As a consequence of Theorem \ref{thm:main_independence}, we obtain:
\begin{theorem}\label{thm:bettinumbers_indepedence_complex}
    Let $G=(V,E)$ be a graph on $n$ vertices, and let $w:V\to \Rea_{\geq 0}$. Then, for all $k\geq 0$,
    \[
\dim(\tilde{H}_k(I(G);\Rea))
    \leq \left|\left\{ I\in\binom{[n]}{k+1} :\, \sum_{i\in I} \eiglargest{i}{\glapw{G}{w}}\geq \sum_{v\in V} w(v)\right\}\right|.
    \]
\end{theorem}

In particular, we obtain the following bound on the homological connectivity of an independence complex.

\begin{corollary}\label{cor:connectivity}
   Let $G=(V,E)$ be a graph,  and let $w:V\to \Rea_{\geq 0}$. Then
   \[
    \eta(I(G))\geq  \min \left\{ m: \, \maxksum{m}{L^{w}(G)}\geq \sum_{v\in V} w(v)\right\}.
   \]
\end{corollary}
Note that Corollary \ref{cor:connectivity} implies Theorem \ref{thm:ABM_connectivity}. Indeed, let $\ell$ be the minimum index such that $\maxksum{\ell}{L(G)}\geq |V|$. Then, since $|V|\leq \maxksum{\ell}{L(G)}\leq \ell\cdot \eiglargest{1}{L(G)}$,  we obtain, by Corollary \ref{cor:connectivity}, $\eta(I(G))\geq \ell \geq |V|/\eiglargest{1}{L(G)}$, recovering the bound in Theorem \ref{thm:ABM_connectivity}.

\begin{example}
    Let $G=(V,E)$ be a matching of size $r$, that is, the union of $r$ disjoint edges. Then $I(G)$ is an $(r-1)$-dimensional sphere, and it can be shown (for example using the formula for the Laplacian spectrum of the join of simplicial complexes; see e.g. \cite[Thm. 2.4]{lew2020spectral2}) that for all $0\leq k\leq r-1$ and $0\leq t\leq k+1$, $2(r-t)=|V|-2t$ is an eigenvalue of $\lap{k}{I(G)}$ with multiplicity $\binom{r}{k+1}\binom{k+1}{t}$. On the other hand, the eigenvalues of $L(G)$ are $0$ and $2$, both with multiplicity $r$, and therefore the multiset $\mathcal{S}_{k+1}(L(G))$ contains, for every $1\leq t\leq k+1$, the element $2t$ repeated $\binom{r}{t}\binom{r}{k+1-t}$ times. It can be then checked that, for all $0\leq k\leq r-1$ and $1\leq i\leq \binom{r}{k+1}+(k+1)\binom{r}{k+1}=(k+2)\binom{r}{k+1}$, we have $\eigsmallest{i}{\lap{k}{I(G)}}= |V|-\ithmaxksum{k+1}{i}{L(G)}$, obtaining equality in the inequality of Theorem \ref{thm:main_independence} in these cases. Moreover, we obtain from Theorem \ref{thm:bettinumbers_indepedence_complex} that $\tilde{H}_k(I(G);\Rea)=0$ for $k\leq r-2$ and $\dim(\tilde{H}_{r-1}(I(G);\Rea))\leq 1$, which are optimal. 
\end{example}

The proof of Theorem \ref{thm:ABM_main} in \cite{aharoni2005eigenvalues} relies on an adaptation of Garland's ``local to global" method (see \cite{garland1973p,ballmann1997l2,zuk1996property}), which in its original form relates between the eigenvalues of high dimensional Laplacians on a simplicial complex $X$ and the eigenvalues of lower dimensional Laplacians associated to certain subcomplexes of $X$. 
Here we follow a different approach. For an $m\times m$ matrix $M$, the \emph{$k$-th additive compound} of $M$ is an $\binom{m}{k}\times \binom{m}{k}$ matrix $M^{[k]}$ whose eigenvalues are all the possible sums of $k$ eigenvalues of $M$. That is, the spectrum of $M^{[k]}$ is exactly $\mathcal{S}_k(M)$ (see Section \ref{sec:compound} for more details). The proof of Theorem \ref{thm:main_independence} follows by determining a close relation between the weighted $k$-dimensional Laplacian of $I(G)$ and the $(k+1)$-th additive compound of its weighted $0$-dimensional Laplacian (which in turn is related to the graph Laplacian $\glapw{G}{w}$).

The homological connectivity of the independence complex $I(G)$ can be bounded from below by various domination and packing parameters of the graph $G=(V,E)$. For example:
\begin{itemize}[leftmargin=*]

\item Let $i\gamma(G)$ be the maximum, over all independent sets $I$ in $G$, of the minimal size of a set $S$ such that every vertex in $I$ is adjacent in $G$ to at least one vertex of $S$.  
Then $\eta(I(G))\geq i\gamma(G)$ (\cite{aharoni2000hall}, see also \cite{meshulam2001clique,meshulam2003domination}).
\item A set $S\subset V$ is called a \emph{neighborhood packing} if the distance in $G$ between every two vertices in $S$ is at least $3$ (equivalently, if the closed neighborhoods of all vertices in $S$ are pairwise disjoint). Let $\rho(G)$ be the maximal size of a neighborhood packing in $G$. It is easy to check that $\rho(G)\leq i\gamma(G)$, and therefore $\eta(I(G))\geq \rho(G)$ (see e.g. \cite{zewi2012vector,barmak2013star}).

\item A function $f:V\to \Rea_{\geq 0}$ is called \emph{star-dominating} (or weakly dominating) if $\deg(v)f(v)+\sum_{u\in N_G(v)} f(u)\geq 1$ for all $v\in V$. Let $\gamma_s^{*}(G)$ be the minimum of $\sum_{v\in V}f(v)$ over all star-dominating functions $f:V\to\Rea_{\geq 0}$. It was shown in \cite{meshulam2003domination} that $\eta(I(G))\geq \gamma_s^{*}(G)$.

\item Let $P:V\to \Rea^{\ell}$ for some $\ell\geq 1$. We say that $P$ is a \emph{vector representation} of $G$ if for all $u,v\in V$ we have $P(u)\cdot P(v)\geq 1$ if $\{u,v\}\in E$ and $P(u)\cdot P(v)\geq 0$ otherwise. A function $f:V\to \Rea_{\geq 0}$ is \emph{dominating for $P$} if $\sum_{v\in V} f(v) P(v)\cdot P(u)\geq 1$ for all $u\in V$. Let $|P|$ be the minimum of $\sum_{v\in V} f(v)$ over all dominating functions for $P$, and let $\Gamma(G)$ be the supremum of $|P|$ over all vector representations of $G$. In \cite{aharoni2005eigenvalues}, the bound $\eta(I(G))\geq \Gamma(G)$ was obtained as an application of Theorem \ref{thm:ABM_connectivity}.

Note that $\Gamma(G)\geq \gamma_s^{*}(G)$ (see \cite{aharoni2005eigenvalues}) and $\Gamma(G)\geq \rho(G)$ (see \cite{zewi2012vector}). This gives alternative proofs for the bounds $\eta(I(G))\geq \gamma_s^{*}(G)$ and $\eta(I(G))\geq \rho(G)$.

\end{itemize}

As a consequence of Theorem \ref{thm:bettinumbers_indepedence_complex} and Corollary \ref{cor:connectivity}, we obtain the following new results:

We say that a function $f:V\to \Rea_{\geq 0}$ is a \emph{fractional quadratic packing} if $\sum_{u\in N_G(v)} f(u)(f(u)+f(v))\leq 1$ for all $v\in V$. Let $\rho_q^{*}(G)$ be the supremum of $\sum_{v\in V} f(v)^2$ over all fractional quadratic packings in $G$. 
Note that the indicator function of a neighborhood packing in $G$ is a fractional quadratic packing, so $\rho_{q}^{*}(G)\geq \rho(G)$.

\begin{theorem}\label{thm:connectivity_from_quadratic_packing}
   Let $G=(V,E)$ be a graph. Then $\eta(I(G))\geq \rho_q^{*}(G)$.
\end{theorem}

 We also obtain the following extension of the bound $\eta(I(G))\geq \Gamma(G)$ from \cite{aharoni2005eigenvalues}.

\begin{theorem}
    \label{thm:homology_from_vector_reps}
Let $G=(V,E)$ be a graph, and let $P:V\to\Rea^{\ell}$ be a vector representation of $G$. Then, for every $f:V\to \Rea_{\geq 0}$, 
\[
\dim(\tilde{H}_k(I(G);\Rea))\leq \left|\left\{ I\in \binom{V}{k+1} :\, \sum_{u\in I} P(u)\cdot \sum_{v\in V}f(v) P(v)  \geq \sum_{v\in V}f(v) \right\}\right|.
\]
\end{theorem}

The paper is organized as follows. In Section \ref{sec:prel} we provide the necessary background on matrix eigenvalues, additive compound matrices and high dimensional Laplacians. Section \ref{sec:main} contains the proof of our main result, Theorem \ref{thm:main_independence}, and its corollaries, Theorem \ref{thm:bettinumbers_indepedence_complex} and Corollary \ref{cor:connectivity}. In Section \ref{sec:applications} we present the proofs of Theorems \ref{thm:connectivity_from_quadratic_packing} and \ref{thm:homology_from_vector_reps}. 
Finally, in Section \ref{sec:additional}, as a simple additional application of additive compound matrices, we present some upper bounds on the sum of the $k$ largest eigenvalues of the Laplacian and adjacency matrices of a graph. 

\section{Preliminaries}\label{sec:prel}

\subsection{Matrix eigenvalues}

The following inequality due to Weyl gives useful bounds for the spectrum of the sum of two symmetric matrices.

\begin{lemma}[See e.g. {\cite[Thm 2.8.1]{brouwer2011spectra}}]\label{lemma:weyl}
Let $A, B$ be real symmetric matrices of size $n\times n$. Then, for all $1\leq i\leq n$,
\[
   \eigsmallest{i}{A+B}\geq \eigsmallest{i}{A}+\eigsmallest{1}{B},
\]
or, equivalently,
\[
 \eiglargest{i}{A+B} \geq \eiglargest{i}{A}+\eiglargest{n}{B}.
\]
\end{lemma}

The following result is known as Cauchy's interlacing theorem:

\begin{theorem}[See e.g. {\cite[Cor. 2.5.2]{brouwer2011spectra}}]\label{thm:cauchy}\label{thm:interlacing}
Let $A$ be a real symmetric matrix of size $n\times n$ and $B$ a principal submatrix of $A$ of size $m\times m$. 
Then, for all $1\leq i\leq m$, 
\[
    \eigsmallest{i}{A} \leq \eigsmallest{i}{B}\leq \eigsmallest{n-m+i}{A}.
\]
\end{theorem}

Finally, we will need the next result, known as Ger{\v s}gorin's circle theorem.

\begin{theorem}[See e.g. {\cite[Theorem 6.1.1]{horn2012matrix}}]
    \label{thm:gersgorin}
  Let $M\in \mathbb{C}^{n\times n}$ and let $\lambda$ be an eigenvalue of $M$. Then, there is some $1\leq i\leq n$ such that
  \[   
    |\lambda-M_{i,i}|\leq \sum_{j\neq i} |M_{j,i}|.
  \]  
  In particular,
  \[
    |\lambda|\leq \max \left\{\sum_{j=1}^n |M_{j,i}| :\, 1\leq i\leq n\right\}.
  \]
\end{theorem}

\subsection{Additive compound matrices}\label{sec:compound}

Let $V$ be an $n$-dimensional vector space over a field $\FF$.  For $k\leq n$, let $\bigwedge^k V$ be the $k$-th exterior power of $V$. Given a linear operator $M:V\to V$, the \emph{ $k$-th additive compound} of $M$ is the linear operator $\sumwedge{k}{M}: \bigwedge^k V\to \bigwedge^k V$ defined by
\[
    \sumwedge{k}{M}(v_1\wedge \cdots \wedge v_k)=\sum_{i=1}^k v_1 \wedge \cdots \wedge (Mv_i )\wedge \cdots \wedge v_k
\]
for every $v_1,\ldots,v_k\in V$. Additive compound operators were studied by Wielandt in \cite{wielandt1967topics} (see also \cite{marshall1979inequalities}). Applications of  additive compounds to differential equations were investigated by
Schwarz in \cite{schwarz1970totally} and London in \cite{london1976derivations} (see also \cite{muldowney1990compound}). In \cite{fiedler1974additive}, Fiedler studied a family of generalized compound operators, interpolating between the classical (multiplicative) compounds and the additive compounds.

A useful property of the operator $\sumwedge{k}{M}$ is the relation between its spectrum and that of $M$:
\begin{theorem}[See e.g. {\cite[Thm. F.5]{marshall1979inequalities}},{\cite[Thm. 2.1]{fiedler1974additive}}]
\label{thm:additive_compound_eigenvalues}
Let $M$ be an $n\times n$ matrix over a field $\mathbb{F}$, with eigenvalues $\lambda_1,\ldots,\lambda_n$. Then, the $k$-th additive compound $\sumwedge{k}{M}$ has eigenvalues $\lambda_{i_1}+\cdots+\lambda_{i_k}$, for $1\leq i_1<\cdots<i_k\leq n$. 
\end{theorem}

Let $e_1,\ldots,e_n$ be the standard basis for $V$. Then $\{e_{i_1}\wedge\cdots\wedge e_{i_k} : \, 1\leq i_1<\cdots <i_k\leq n\}$ is a basis for the exterior power $\bigwedge^k V$. Identifying $\sumwedge{k}{M}$ with its matrix representation with respect to this basis, we obtain:

\begin{theorem}[See e.g. {\cite[Thm. 2.4]{fiedler1974additive}}]\label{thm:additive_compound_matrix}
  Let $M$ be an $n\times n$ matrix, and let $1\leq k\leq n$. Then, $\sumwedge{k}{M}$ is an $\binom{n}{k}\times\binom{n}{k}$ matrix, with rows and columns indexed by the $k$-subsets of $[n]$, defined by
   \[
    (\sumwedge{k}{M})_{\sigma,\tau}=\begin{cases}
        \sum_{i\in\sigma} M_{i,i} & \text{ if } \sigma=\tau,\\
        (-1)^{\epsilon(\sigma,\tau)} M_{i,j} & \text{ if } |\sigma\cap\tau|=k-1,\, \sigma\setminus\tau=\{i\},\, \tau\setminus\sigma=\{j\},\\
        0 &\text{ otherwise,}
    \end{cases}
    \]   
for every $\sigma,\tau\in\binom{[n]}{k}$, where, for $|\sigma\cap\tau|=k-1$, $ \sigma\setminus\tau=\{i\}$, $\tau\setminus\sigma=\{j\},$  $\epsilon(\sigma,\tau)$ denotes the number of elements in $\sigma\cap \tau$ between $i$ and $j$.  
\end{theorem}

\subsection{High dimensional Laplacians}\label{sec:higher_laplacians}

Let $X$ be a simplicial complex on vertex set $[n]$. An element $\sigma\in X$ is called a \emph{face} or \emph{simplex} of $X$. The \emph{dimension} of $\sigma\in X$ is defined as $\dim(\sigma)=|\sigma|-1$. For every $k\geq -1$, let $X(k)$ be the collection of all $k$-dimensional faces of $X$, and denote $f_k(X)=|X(k)|$.
For a simplex $\sigma\in X$, let $N_X(\sigma)=\{v\in V\setminus \sigma:\, \sigma\cup\{v\}\in X\}$.

For $\sigma=\{i_0,\ldots,i_k\}\in X$, where $1\leq i_0<i_1<\ldots<i_k\leq n$, let $e_{\sigma}=e_{i_0}\wedge \cdots \wedge e_{i_k}\in \bigwedge^{k+1}\Rea^{n}$.
For $k\geq -1$, the space of \emph{$k$-cochains} on $X$, denoted by $C^k(X;\Rea)$,  is the subspace of $\bigwedge^{k+1} \Rea^{n}$ spanned by the elements $\{e_{\sigma}:\, \sigma\in X(k)\}$. 
We will call this spanning set the \emph{standard basis} for $C^k(X;\Rea)$.
The \emph{$k$-th coboundary operator} is the linear operator $d_k:C^k(X;\Rea)\to C^{k+1}(X;\Rea)$ acting on standard basis elements by
\[
   d_k (e_{i_0}\wedge \cdots \wedge e_{i_k})= \sum_{j\in N_X(\{i_0,\ldots,i_k\})}
   e_{i_0}\wedge \cdots \wedge e_{i_k}\wedge e_j.
\]
The \emph{$k$-th (reduced) cohomology group} of $X$ is then defined as
\[
    \tilde{H}^k(X;\Rea)= \Ker d_k / \Ima d_{k-1}.
\]
It follows from the universal coefficient theorem (and also by more elementary arguments) that $\tilde{H}^k(X;\Rea)$ is isomorphic to $ \tilde{H}_k(X;\Rea)$, the $k$-th homology group of $X$.

Let $w: X\to \Rea_{\geq 0}$ be a weight function on $X$. We say that $w$ is positive if $w(\sigma)>0$ for all $\sigma\in X$. A positive weight function on $X$ induces an inner product on each $C^k(X;\Rea)$, defined by
\[
 \langle e_{\sigma},e_{\tau}\rangle= \begin{cases} w(\sigma) & \text{ if } \sigma=\tau,\\
 0 & \text{ otherwise,}
 \end{cases}
\]
for all $\sigma,\tau\in X(k)$.
Let $d_k^*:C^{k+1}(X;\Rea)\to C^{k}(X;\Rea)$ be the adjoint of $d_k$ with respect to the inner products induced by $w$. The $w$-weighted \emph{$k$-dimensional Laplacian operator} on $X$ is defined as 
\[
    \lapw{k}{X}{w}=  d_k^* d_k + d_{k-1} d_{k-1}^* : C^k(X;\Rea)\to C^k(X;\Rea).
\]
Note that $\lapw{k}{X}{w}$ is a positive semi-definite operator.  
The simplicial Hodge theorem, observed by Eckmann in \cite{eckmann1944harmonische}, states that $\Ker \lapw{k}{X}{w}$ is isomorphic to the $k$-th cohomology group $\tilde{H}^k(X;\Rea)$ (and therefore also to the $k$-th homology group $\tilde{H}_k(X;\Rea)$). We can restate this as:
\begin{theorem}[Eckmann \cite{eckmann1944harmonische}]\label{thm:eckmann}
    Let $X$ be a simplicial complex, and let $w:X\to\Rea_{>0}$. Then, for all $k\geq 0$ and $0\leq j\leq f_k(X)-1$, $\dim(\tilde{H}_k(X;\Rea))\leq j$ if and only if $\eigsmallest{j+1}{\lapw{k}{X}{w}}>0$.
\end{theorem}
In particular, we have $\tilde{H}_k(X;\Rea)=0$ if and only if $\eigsmallest{1}{\lapw{k}{X}{w}}>0$.
Identifying $\lapw{k}{X}{w}$ with its matrix representation with respect to the standard basis, we obtain the following explicit description of $\lapw{k}{X}{w}$.

\begin{theorem}[\cite{horak2013spectra}, see also {\cite[Eq. 3.4]{duval2002shifted}}]\label{thm:k_laplacian_matrix}
Let $X$ be a simplicial complex on vertex set $[n]$, and let $w:X\to\Rea_{>0}$. Then, for all $k\geq -1$, $\lapw{k}{X}{w}$ is an $f_k(X)\times f_k(X)$ matrix, with rows and columns indexed by the $k$-dimensional simplices of $X$, defined by
\[
    \lapw{k}{X}{w}_{\sigma,\tau}= \begin{dcases}
        \sum_{u\in N_X(\sigma)} \frac{w(\sigma\cup\{u\})}{w(\sigma)} + \sum_{v\in\sigma} \frac{w(\sigma)}{w(\sigma\setminus \{v\})}  & \text{ if } \sigma=\tau,\\
        (-1)^{\epsilon(\sigma,\tau)} \frac{w(\tau)}{w(\sigma\cap \tau)} & \text{ if } |\sigma\cap\tau|=k,\, \sigma\cup\tau\notin X,\\
        (-1)^{\epsilon(\sigma,\tau)}\left( \frac{w(\tau)}{w(\sigma\cap \tau)} - \frac{w(\sigma\cup\tau)}{w(\sigma)} \right) & \text{ if } |\sigma\cap\tau|=k,\, \sigma\cup\tau\in X,\\
        0 & \text{otherwise,}
    \end{dcases} 
\]
for every $\sigma,\tau\in X(k)$, where for $|\sigma\cap\tau|=k$, $ \sigma\setminus\tau=\{i\}$, $\tau\setminus\sigma=\{j\},$  $\epsilon(\sigma,\tau)$ denotes the number of elements in $\sigma\cap \tau$ between $i$ and $j$. 
\end{theorem}

We will focus here on a special class of weight functions. Let $w:[n]\to \Rea_{> 0}$. We extend $w$ to all faces of $X$ by defining
\[
    w(\sigma) = \prod_{v\in \sigma} w(v)
\]
for every $\emptyset\ne \sigma\in X$, and $w(\emptyset)=1$.
In this case, we obtain from Theorem \ref{thm:k_laplacian_matrix} the following representation for $\lapw{k}{X}{w}$.
\begin{lemma}
    \label{lemma:vertex_weighted_k_lap}
Let $X$ be a simplicial complex on vertex set $[n]$, and let $w:[n]\to\Rea_{> 0}$. Then, for all $k\geq -1$, $\lapw{k}{X}{w}$ is an $f_k(X)\times f_k(X)$ matrix, with rows and columns indexed by the $k$-dimensional simplices of $X$, defined by
\[
    \lapw{k}{X}{w}_{\sigma,\tau}= \begin{dcases}
        \sum_{u\in N_X(\sigma)} w(u) + \sum_{v\in\sigma} w(v) & \text{ if } \sigma=\tau,\\
          (-1)^{\epsilon(\sigma,\tau)} w(v) & \text{if } |\sigma\cap\tau|=k, \,\tau\setminus\sigma=\{v\},\,
        \sigma\cup\tau\notin X,
        \\
        0 & \text{otherwise,}
    \end{dcases} 
\]
for every $\sigma,\tau\in X(k)$, where for $|\sigma\cap\tau|=k$, $ \sigma\setminus\tau=\{i\}$, $\tau\setminus\sigma=\{j\},$  $\epsilon(\sigma,\tau)$ denotes the number of elements in $\sigma\cap \tau$ between $i$ and $j$. 
\end{lemma}

We can extend the definition of the vertex-weighted $k$-dimensional Laplacian to non-negative weight functions $w:[n]\to\Rea_{\geq 0}$, by letting $\lapw{k}{X}{w}$ be defined as in Lemma \ref{lemma:vertex_weighted_k_lap}. In this case, the following holds.

\begin{lemma}\label{lemma:homology_from_eigenvalues_degenerate_case}
Let $X$ be a simplicial complex on vertex set $[n]$, and let $w:[n]\to\Rea_{\geq 0}$. Let $k\geq 0$ and $0\leq j\leq f_k(X)-1$. If $\eigsmallest{j+1}{\lapw{k}{X}{w}}>0$, then  $\dim (\tilde{H}_k(X;\Rea))\leq j$.
\end{lemma}
\begin{proof}
    For every $\epsilon>0$, we define a positive weight function $w_{\epsilon}:[n]\to \Rea_{>0}$ by
    \[
        w_{\epsilon}(v)=\begin{cases}
                    w(v) & \text{if } w(v)>0,\\
                    \epsilon & \text{if } w(v)=0.
             \end{cases} 
    \]
    By Lemma \ref{lemma:vertex_weighted_k_lap}, $\lapw{k}{X}{w}$ is the limit of $\lapw{k}{X}{w_{\epsilon}}$ as $\epsilon\to 0$. By the continuity of eigenvalues, the eigenvalues of $\lapw{k}{X}{w}$ are real numbers and satisfy $\eigsmallest{i}{\lapw{k}{X}{w}}=\lim_{\epsilon\to0} \eigsmallest{i}{\lapw{k}{X}{w_{\epsilon}}}$ for all $1\leq i\leq f_k(X)$. Assume $\eigsmallest{j+1}{\lapw{k}{X}{w}}>0$ for some $0\leq j\leq f_k(X)-1$. Then, for small enough $\epsilon$, we have $\eigsmallest{j+1}{\lapw{k}{X}{w_{\epsilon}}}>0$. Since $w_{\epsilon}$ is positive, we obtain from Theorem \ref{thm:eckmann} that $\tilde{H}_k(X;\Rea)\leq j$.
\end{proof}

\section{Main results}\label{sec:main}

In this section we prove our main result, Theorem \ref{thm:main_independence}, and its corollaries, Theorem \ref{thm:bettinumbers_indepedence_complex} and Corollary \ref{cor:connectivity}.
Let $G=(V,E)$ be a graph. The \emph{clique complex} (also know as the flag complex) of $G$ is the simplicial complex $X(G)$ on vertex set $V$ whose simplices are the cliques of $G$ (i.e. vertex subsets forming a complete subgraph). Note that for every graph $G$, $I(G)=X(\bar{G})$, where $\bar{G}$ is the graph complement of $G$.
We will need the following simple lemma about sums of ``weighted degrees" in a clique complex (see \cite[Claim 3.4]{aharoni2005eigenvalues} for a similar result in the unweighted setting).

\begin{lemma}\label{lemma:degree_sum}
    Let $G=(V,E)$ be a graph and let 
     $X=X(G)$. Let $w:V\to \Rea_{\geq 0}$.  Then, for all $k\geq 0$ and $\sigma\in X(k)$,
\[
     \left(\sum_{v\in \sigma} \sum_{u\in N_G(v)} w(u)\right)-\sum_{v\in N_X(\sigma)}w(v)\leq k\sum_{v\in V} w(v).
\]
\end{lemma}
\begin{proof}

    Let $\sigma\in X(k)$.
    Note that, since $X$ is a clique complex, we have $N_X(\sigma)=\bigcap_{v\in\sigma}N_G(v)$, and therefore
    \[
   \sum_{v\in \sigma} \sum_{u\in N_G(v)} w(u) \leq (k+1)\sum_{v\in N_X(\sigma)} w(v)+k\sum_{v\in V\setminus N_X(\sigma)} w(v).
    \]
Hence, we obtain
\begin{align*}
      \left(\sum_{v\in \sigma} \sum_{u\in N_G(v)} w(u)\right)-\sum_{v\in N_X(\sigma)}w(v)
     &\leq 
      k\sum_{v\in N_X(\sigma)} w(v)+k\sum_{v\in V\setminus N_X(\sigma)} w(v)
     \\ &=k\sum_{v\in V} w(v).
\end{align*}
\end{proof}

\begin{theorem}
\label{thm:main}
    Let $G=(V,E)$ be a graph, and let $w:V\to \Rea_{\geq 0}$. Then, for every  $k\geq 0$ and $1\leq i\leq f_k(X(G))$,
    \[
    \eigsmallest{i}{\lapw{k}{X(G)}{w}}\geq \ithminksum{k+1}{i}{\lapw{0}{X(G)}{w}}-k\sum_{v\in V} w(v).
    \]
\end{theorem}
\begin{proof}
Without loss of generality, assume $V=[n]$.
    Let $X=X(G)$.
    Let $L$ be the principal submatrix of $\sumwedge{k+1}{\lapw{0}{X}{w}}$ obtained by removing all rows and columns except those corresponding to $k$-dimensional faces of $X$.
    By Lemma \ref{lemma:vertex_weighted_k_lap}, for all $u,v\in V$,
    \begin{equation}\label{eq:0lap}
        \lapw{0}{X}{w}_{u,v}=\begin{cases}
            w(u)+\sum_{u'\in N_G(u)}w(u') & \text{ if } u=v,\\
            w(v) & \text{ if } \{u,v\}\notin E,\\
            0 & \text{ otherwise}.
        \end{cases}
    \end{equation}
        Let $\sigma,\tau\in X(k)$ with $|\sigma\cap\tau|=k$. Let $u,v$ be the two vertices in the symmetric difference $\sigma\triangle \tau$. Note that, since $X$ is a flag complex, we have $\sigma\cup\tau\in X$ if and only if $\{u,v\}\in E$.
    Thus, by Theorem \ref{thm:additive_compound_matrix}, we have
    \begin{align*}
        L_{\sigma,\tau} &=
        \begin{cases}
            \sum_{v\in \sigma} \left( w(v)+\sum_{u\in N_G(v)} w(u)\right) & \text{ if } \sigma=\tau,\\
            (-1)^{\epsilon(\sigma,\tau)}w(v) & \text{ if } |\sigma\cap \tau|=k, \tau\setminus\sigma=\{v\},\, \sigma\triangle\tau\notin E, \\
            0 & \text{ otherwise}
            \end{cases}
        \\
        &=\begin{cases}
            \sum_{v\in \sigma} \sum_{u\in N_G(v)} w(u) +\sum_{v\in \sigma}w(v) & \text{ if } \sigma=\tau,\\
            (-1)^{\epsilon(\sigma,\tau)}w(v) & \text{ if } |\sigma\cap \tau|=k, \tau\setminus\sigma=\{v\},\,\sigma\cup\tau\notin X,\\
            0 & \text{ otherwise,}
        \end{cases}
    \end{align*}
    for all $\sigma,\tau\in X(k)$. Let $R$ be the $f_k(X)\times f_k(X)$ diagonal matrix with elements
    \[
        R_{\sigma,\sigma}=\sum_{v\in N_X(\sigma)}w(v)-\sum_{v\in \sigma} \sum_{u\in N_G(v)}w(u) 
    \]
    for every $\sigma\in X(k)$. By Lemma \ref{lemma:vertex_weighted_k_lap}, we have
    \[
       \lapw{k}{X}{w}=L+R.
    \]
    By Lemma \ref{lemma:degree_sum}, $R_{\sigma,\sigma}\geq -k\sum_{v\in V}w(v)$ for all $\sigma\in X(k)$, and therefore $\eigsmallest{1}{R}\geq -k\sum_{v\in V} w(v)$. By Theorem \ref{thm:interlacing} and Theorem \ref{thm:additive_compound_eigenvalues}, we obtain for every $1\leq i\leq f_k(X)$,
    \[
        \eigsmallest{i}{L} \geq \eigsmallest{i}{\sumwedge{k+1}{\lapw{0}{X}{w}}} = \ithminksum{k+1}{i}{\lapw{0}{X}{w}}.
    \]
    Thus, by Lemma \ref{lemma:weyl}, 
    \[
    \eigsmallest{i}{\lapw{k}{X}{w}}\geq \eigsmallest{i}{L}+\eigsmallest{1}{R} \geq \ithminksum{k+1}{i}{\lapw{0}{X}{w}}- k\sum_{v\in V} w(v).
    \]
\end{proof}

Theorem \ref{thm:main_independence} now follows by applying Theorem \ref{thm:main} to the complement of the graph $G$.

\begin{proof}[Proof of Theorem \ref{thm:main_independence}]
    Let $\bar{G}$ be the complement graph of $G$. 
    Note that $I(G)=X(\bar{G})$.
By \eqref{eq:glap_def} and \eqref{eq:0lap}, we have 
\[
\lapw{0}{I(G)}{w}=\lapw{0}{X(\bar{G})}{w}=\left(\sum_{v\in V} w(v)\right)I- \glapw{G}{w},
\]
where $I$ is the $|V|\times |V|$ identity matrix.
So, $\eigsmallest{i}{\lapw{0}{I(G)}{w}}=\left(\sum_{v\in V} w(v)\right)-\eiglargest{i}{\glapw{G}{w}}$ for all $1\leq i\leq |V|$, and therefore
    \[
    \ithminksum{k+1}{i}{\lapw{0}{I(G)}{w}}=(k+1)\left(\sum_{v\in V} w(v)\right)-\ithmaxksum{k+1}{i}{\glapw{G}{w}}
    \]
for all $1\leq i\leq \binom{|V|}{k+1}$. Hence, by Theorem \ref{thm:main}, we obtain for all $1\leq i\leq f_k(I(G))$,
\begin{align*}
\eigsmallest{i}{\lapw{k}{I(G)}{w}}
&
\geq \ithminksum{k+1}{i}{\lapw{0}{I(G)}{w}}-k\sum_{v\in V} w(v)
\\
&= (k+1)\left(\sum_{v\in V}w(v)\right) - \ithmaxksum{k+1}{i}{\glapw{G}{w}} - k\sum_{v\in V} w(v)
\\
& = \left(\sum_{v\in V} w(v)\right)-\ithmaxksum{k+1}{i}{L^{w}(G)}.
\end{align*}
\end{proof}

\begin{proof}[Proof of Theorem \ref{thm:bettinumbers_indepedence_complex}]
Let
    \begin{align*}
j&=\left|\left\{ I\in\binom{[n]}{k+1} :\, \sum_{i\in I} \eiglargest{i}{\glapw{G}{w}}\geq \sum_{v\in V} w(v)\right\}\right|
\\&=\max\left(\left\{ 1\leq i\leq \binom{n}{k+1} :\, \ithmaxksum{k+1}{i}{L^{w}(G)}\geq \sum_{v\in V} w(v)\right\}\cup\{0\}\right).
    \end{align*}
If $j=\binom{n}{k+1}$, then we have $\dim(\tilde{H}_k(I(G);\Rea))\leq f_k(I(G))\leq j$ as wanted. Otherwise, by the maximality of $j$, we have
$\ithmaxksum{k+1}{j+1}{\glapw{G}{w}}<\sum_{v\in V}w(v)$, and therefore, by Theorem \ref{thm:main_independence}, $\eigsmallest{j+1}{\lapw{k}{I(G)}{w}}>0$. Hence, by Lemma \ref{lemma:homology_from_eigenvalues_degenerate_case}, $\dim(\tilde{H}_k(I(G);\Rea))\leq j$.
\end{proof}

\begin{proof}[Proof of Corollary \ref{cor:connectivity}]
Recall that we defined the homological connectivity of $I(G)$ as
$
\eta(I(G))=\max\{ k: \tilde{H}_{i}(I(G);\Rea)=0 \text{ for all } i\leq k-2\}.
$
    Let $k= \min \{ m: \, \maxksum{m}{L(G)}\geq \sum_{v\in V} w(v)\}$. If $k=1$ then, since  $\tilde{H}_{-1}(I(G))=0$, we have $\eta(I(G))\geq 1=k$, as wanted. Otherwise, assume $k>1$. Then, $\maxksum{m}{L(G)}<\sum_{v\in V} w(v)$ for $m\leq k-1$, and therefore, 
    by Theorem \ref{thm:bettinumbers_indepedence_complex}, $\tilde{H}_{i}(I(G);\Rea)=0$ for $i\leq k-2$. Thus, we obtain $\eta(I(G))\geq k$.
\end{proof}

\section{Domination and packing parameters}
\label{sec:applications}

In this section we prove Theorems \ref{thm:connectivity_from_quadratic_packing} and \ref{thm:homology_from_vector_reps}, relating the homology of the independence complex $I(G)$ to different packing and domination parameters of $G$.

Recall that a function $f:V\to \Rea_{\geq 0}$ is called a fractional quadratic packing if $\sum_{u\in N_G(v)} f(u)(f(u)+f(v))\leq 1$ for all $v\in V$, and $\rho_q^{*}(G)$ is defined as the supremum of $\sum_{v\in V} f(v)^2$ over all fractional quadratic packings of $G$.

\begin{proof}[Proof of Theorem \ref{thm:connectivity_from_quadratic_packing}]
    Let $f:V\to \Rea_{\geq 0}$ be a fractional quadratic packing of $G$, and let $w:V\to \Rea_{\geq 0}$ be defined by $w(v)=f(v)^2$ for all $v\in V$.
   By Ger{\v s}gorin's Theorem (Theorem \ref{thm:gersgorin}) and \eqref{eq:simlap}, we have
   \begin{align*}
   \eiglargest{1}{\simlapw{G}{w}} &\leq 
    \max \left\{ \sum_{u\in V}\left|\simlapw{G}{w}_{u,v}\right| : v\in V\right\}
\\
  &=
   \max \left\{ \sum_{u\in N_G(v)}f(u)^2 +\sum_{u\in N_G(v)} f(u)f(v) : v\in V\right\}
\\
&=
   \max \left\{ \sum_{u\in N_G(v)}f(u)(f(u)+f(v)) : v\in V\right\} \leq 1.
   \end{align*}
Thus, $\maxksum{m}{\simlapw{G}{w}}\leq m$ for all $m$. By Corollary \ref{cor:connectivity} (recall that the eigenvalues of $\simlapw{G}{w}$ are the same as those of $\glapw{G}{w}$), we obtain
\begin{align*}
       \eta(I(G))&\geq  \min \left\{ m: \, \maxksum{m}{\simlapw{G}{w}}\geq \sum_{v\in V} f(v)^2\right\}
    \\
&\geq \min \left\{ m: \, m\geq \sum_{v\in V} f(v)^2\right\}=\left\lceil \sum_{v\in V}f(v)^2\right\rceil.
\end{align*}
Since this holds for all fractional quadratic packings of $G$, we obtain $\eta(I(G))\geq \lceil \rho^*_q(G)\rceil\geq \rho^*_q(G)$, as wanted.
\end{proof}
\begin{remark}[Remarks.]
    \begin{enumerate}[leftmargin=*]
        \item By duality of linear programming, $\gamma^*_s(G)$ is the maximum of $\sum_{v\in V}f(v)$ over all $f:V\to\Rea_{\geq0}$ satisfying $\deg(v)f(v)+\sum_{u\in N_G(v)} f(u)\leq 1$ for all $v\in V$. Applying Ger{\v s}gorin's Theorem to the matrix $\glapw{G}{f}$, and following essentially the same arguments as in the proof of Theorem \ref{thm:connectivity_from_quadratic_packing}, we obtain a new proof of the bound $\eta(I(G))\geq \gamma_s^{*}(G)$.

        \item Let $C_n$ be the cycle graph on $n$ vertices. Assume its vertex set is $[n]$, and define $f:[n]\to \Rea_{\geq 0}$ by
        \[
            f(i)=\begin{cases}
                0 & \text{if } i \equiv 1 \mod 3,\\
                \frac{1}{\sqrt{2}} & \text{if } i\equiv 2, 0 \mod 3.
            \end{cases} 
        \]
        It is easy to check that $f$ is a fractional quadratic packing, and satisfies $\sum_{i=1}^n f(i)^2 = k$ if $n=3k$ or $n=3k+1$ and $\sum_{i=1}^n f(i)^2 = k+1/2$ if $n=3k+2$ for some $k$. Therefore, we obtain
        \[
            \rho^*_q(C_n)\geq \begin{cases}
                \lfrac{n}{3} & \text{if } n\equiv 0,1 \mod 3,\\
                \lfrac{n}{3}+\frac{1}{2} & \text{if } n\equiv 2 \mod 3.
            \end{cases}
        \]
        Note that this lower bound coincides with the lower bound obtained for $\Gamma(C_n)$ in \cite{aharoni2005eigenvalues}.
        By Theorem \ref{thm:connectivity_from_quadratic_packing}, we obtain $\eta(I(C_n))\geq \lceil \rho^*_q\rceil =\lfloor (n+1)/3\rfloor$, which is tight for all $n$ (see \cite[Claim 3.3]{meshulam2001clique} or \cite[Prop. 5.2]{kozlov1999complexes}), and is better than the bounds obtained from $\gamma_s^*(C_n)=n/4$ (see \cite{aharoni2005eigenvalues}) or $\rho(C_n)=\lfloor n/3 \rfloor$.
    \end{enumerate}
\end{remark}

\begin{proof}[Proof of Theorem \ref{thm:homology_from_vector_reps}]
    Let $M\in \Rea^{V\times V}$ be the matrix defined by
    \[
       M_{u,v}=\begin{cases}
           P(u)\cdot \sum_{u'\neq u} f(u')P(u') & \text{if } u=v,\\
           -\sqrt{f(u)f(v)}P(u)\cdot P(v) & \text{if } u\neq v
       \end{cases}
    \]
    for all $u,v\in V$.
    We may think of $f:V\to\Rea_{\geq 0}$ as a weight function on $V$, and consider the vertex-weighted Laplacian matrix $\simlapw{G}{f}$. For all $x\in \Rea^V$, we have
    \begin{align}\label{eq:abm}
        x^T \simlapw{G}{f} x 
        &= \sum_{\{u,v\}\in E} \left(\sqrt{f(v)}x_u-\sqrt{f(u)}x_v\right)^2
   \nonumber  \\   
       &\leq \frac{1}{2} \sum_{(u,v)\in V\times V} \left(\sqrt{f(v)}x_u-\sqrt{f(u)}x_v\right)^2 P(u)\cdot P(v)
   \nonumber  \\
    &=\sum_{u\in V} \left(P(u)\cdot\sum_{v\neq u} f(v) P(v)\right) x_u^2 
   \nonumber  \\ &- \sum_{\substack{(u,v)\in V\times V\\ u\neq v}}\left(\sqrt{f(u)f(v)}P(u)\cdot P(v)\right) x_u x_v = x^T M x.
        \end{align}

    Let $Q=M-\simlapw{G}{f}$. By \eqref{eq:abm}, $Q$ is positive semi-definite.
    Let $\tilde{P}\in \Rea^{\ell\times |V|}$ be the matrix whose columns are the vectors $\{\sqrt{f(v)} P(v)\}_{v\in V}$. Then $\tilde{P}^T \tilde{P}$ is a positive semi-definite matrix satisfying
    \[
    (\tilde{P}^T \tilde{P})_{u,v}= \sqrt{f(u)f(v)}P(u)\cdot P(v)
    \]
    for all $u,v\in V$. Let $T\in \Rea^{V\times V}$ be the diagonal matrix with elements $T_{u,u}=P(u)\cdot \sum_{v\in V}f(v) P(v)$. Note that $T=M+\tilde{P}^T \tilde{P}$, so that $T=\simlapw{G}{f}+(\tilde{P}^T \tilde{P}+Q)$. Let $|V|=n$. Since $\tilde{P}^T \tilde{P}+Q$ is positive semi-definite, we have 
    $\eiglargest{n}{\tilde{P}^T \tilde{P}+Q)}\geq 0$.
    Therefore, by Lemma \ref{lemma:weyl}, we obtain for all $1\leq i\leq n$, 
    \[
    \eiglargest{i}{T} \geq \eiglargest{i}{\simlapw{G}{f}}+\eiglargest{n}{\tilde{P}^T \tilde{P}+Q)} \geq \eiglargest{i}{\simlapw{G}{f}}.
    \]
    Since $T$ is a diagonal matrix, its eigenvalues are $\{P(u)\cdot\sum_{v\in V} f(v) P(v)\}_{u\in V}$. Thus, by Theorem \ref{thm:bettinumbers_indepedence_complex} (recall that the eigenvalues of $\simlapw{G}{f}$ are the same as those of $\glapw{G}{f}$), we obtain
    \begin{align*}
        \dim(\tilde{H}_k(I(G);\Rea))&\leq
        \left|\left\{ I\in\binom{[n]}{k+1} :\, \sum_{i\in I} \eiglargest{i}{\simlapw{G}{f}}\geq \sum_{v\in V} f(v)\right\}\right|
       \\
       &\leq 
        \left|\left\{ I\in\binom{[n]}{k+1} :\, \sum_{i\in I} \eiglargest{i}{T}\geq \sum_{v\in V} f(v)\right\}\right|
       \\
        &=
        \left|\left\{ I\in \binom{V}{k+1} :\, \sum_{u\in I} P(u)\cdot \sum_{v\in V}f(v) P(v)  \geq \sum_{v\in V}f(v) \right\}\right|.
    \end{align*}
    
\end{proof}

\begin{remark}[Remarks.]
\begin{enumerate}[leftmargin=*]
    \item
Let $P:V\to \Rea^{\ell}$ be a vector representation of $G$. We say that a function $f:V\to \Rea_{\geq 0}$ satisfying $\sum_{v\in V} f(v) P(v)\cdot P(u)\leq 1$ for all $u\in V$ is \emph{dually dominating} for $P$.
By duality of linear programming, $|P|$ is the maximum of $\sum_{v\in V} f(v)$ over all dually dominating functions.

Let $f:V\to\Rea_{\geq 0}$ be a dually dominating function for $P$ satisfying $|P|=\sum_{v\in V} f(v)$. Let $k\leq \lceil |P| \rceil -2$.
Then, for all $I\in\binom{V}{k+1}$, we have 
\[
\sum_{u\in I} P(u)\cdot \sum_{v\in V} f(v) P(v) \leq k+1 \leq \lceil |P| \rceil -1< |P|=\sum_{v\in V} f(v).
\]
By Theorem \ref{thm:homology_from_vector_reps}, we obtain
$
\tilde{H}_k(I(G);\Rea)= 0.
$
Hence, $\eta(I(G))\geq |P|$. 
Since this holds for all vector representations $P$ of $G$, we recover the bound $\eta(I(G))\geq \Gamma(G)$.

\item The proof of Theorem \ref{thm:homology_from_vector_reps}  closely follows the arguments in \cite{aharoni2005eigenvalues}, in particular the proof of \cite[Claim 4.2]{aharoni2005eigenvalues}. The argument in \cite{aharoni2005eigenvalues} relies on the application of Theorem \ref{thm:ABM_connectivity} to a graph $G'$ obtained from $G$ by replacing each vertex $v$ in $G$ by an independent set (of size depending on the value of $f(v)$, where $f$ is a dually dominating function for $P$).
One main difference in our proof is that the use of weighted Laplacian matrices allows us to eliminate the need for this ``duplication of vertices" argument (and indeed this was our motivation for the study of vertex-weighted Laplacians).
\end{enumerate}
\end{remark}

The following result can be obtained as a consequence of Theorem \ref{thm:homology_from_vector_reps} (using the vector representation introduced in \cite[Theorem 4.2]{zewi2012vector}). We include here, however, a simple direct proof using Theorem \ref{thm:bettinumbers_indepedence_complex}.

\begin{proposition}\label{prop:homology_from_packing}
    Let $S\subset V$ be a neighborhood packing in $G$. Let $\deg(S)=\sum_{v\in S}\deg(v)$. Then, for all $k\geq 0$
    \[
    \dim(\tilde{H}_k(I(G);\Rea))\leq \sum_{m=|S|}^{k+1}\binom{\deg(S)}{m}\binom{|V|-\deg(S)}{k+1-m}.
    \]
\end{proposition}
    In particular, letting $S$ be a maximal neighborhood packing, we obtain $\tilde{H}_k(I(G);\Rea)=0$ for $k\leq \rho(G)-2$, recovering the bound $\eta(I(G))\geq \rho(G)$. Note that if $G$ is a matching of size $r$, then, letting $S$ consist of one vertex from each edge in $G$, we obtain from Proposition \ref{prop:homology_from_packing} the tight bounds $\tilde{H}_{k}(I(G);\Rea)=0$ for $k\leq r-2$ and $\dim(\tilde{H}_{r-1}(I(G);\Rea))\leq 1$.

\begin{proof}[Proof of Proposition \ref{prop:homology_from_packing}]
Let $|V|=n$.
   Let $w:V\to \Rea_{\geq 0}$ be defined by
   \[
   w(v)=\begin{cases}
       1 & \text{if } v\in S,\\
       0 & \text{otherwise.}
   \end{cases}
   \]
   Since no two vertices in $S$ are adjacent in $G$, $\simlapw{G}{w}$ is a diagonal matrix. Moreover, since every vertex in $V$ is adjacent to at most one vertex in $S$, we have
   \[
   \simlapw{G}{w}_{u,u}=\begin{cases}
                1 & \text{if } u\in \bigcup_{v\in S}N_G(v),\\
                0 & \text{otherwise,}
              \end{cases} 
   \]
   for all $u\in V$. Therefore, the eigenvalues of $\simlapw{G}{w}$ are $1$ with multiplicity $\deg(S)$ and $0$ with multiplicity $|V|-\deg(S)$. Since $\sum_{v\in V} w(v)= |S|$, we obtain by Theorem \ref{thm:bettinumbers_indepedence_complex},
    \begin{align*}
        \dim(\tilde{H}_k(I(G);\Rea))&\leq
        \left|\left\{ I\in\binom{[n]}{k+1} :\, \sum_{i\in I} \eiglargest{i}{\simlapw{G}{w}}\geq |S|\right\}\right|
       \\
       &= 
        \left|\left\{ I\in\binom{[n]}{k+1} :\, 
        |I \cap \{1,\ldots,\deg(S)\}|\geq |S|
        \right\}\right|
       \\
        &=
        \sum_{m=|S|}^{k+1} \binom{\deg(S)}{m}\binom{|V|-\deg(S)}{k+1-m}.
    \end{align*}
\end{proof}

\section{An additional application}\label{sec:additional}

The following inequality due to Merris follows immediately by applying Ger{\v s}gorin's circle theorem (Theorem \ref{thm:gersgorin}) to the $k$-th additive compound of a matrix $M$:

\begin{theorem}[Merris \cite{merris1994inequality}]\label{thm:merris_gersgorin}
    Let $M\in\mathbb{C}^{n\times n}$ be an hermitian matrix. Then
    \[
    \maxksum{k}{M} \leq \max_{\sigma\in\binom{[n]}{k}} \left( \sum_{i\in\sigma}M_{i,i}+\sum_{\substack{i\in\sigma,\\j\notin\sigma}}|M_{i,j}|\right).
    \]
\end{theorem}

Recall that the \emph{adjacency matrix} of a graph $G=(V,E)$ is the $|V|\times|V|$ matrix $A(G)$ defined by
\[
A(G)_{u,v}=\begin{cases} 1 & \text{ if } \{u,v\}\in E,\\
    0 & \text{ otherwise.}
    \end{cases}
\]
By applying Theorem \ref{thm:merris_gersgorin} to the Laplacian and adjacency matrices of a graph $G$, we obtain:

\begin{proposition}
    Let $G=(V,E)$ be a graph, and let $1\leq k\leq |V|$. Then,
    \[
    \maxksum{k}{L(G)}\leq 2\cdot \max_{\sigma\in\binom{V}{k}} \left|\left\{ e\in E:\, e\cap \sigma\neq \emptyset\right\}\right|,
    \]
    and
    \[
    \maxksum{k}{A(G)}\leq \max_{\sigma\in\binom{V}{k}}  \left|\left\{ e\in E:\, |e\cap \sigma|=1 \right\}\right|.
    \]
\end{proposition}
\begin{proof}
    By Theorem \ref{thm:merris_gersgorin}, we have
\[
    \maxksum{k}{L(G)} \leq \max_{\sigma\in\binom{V}{k}} \left( \sum_{v\in\sigma}\deg(v) + \sum_{\substack{v\in\sigma,u\notin \sigma:\\ \{u,v\}\in E}} 1\right).
    \]
Note that for all $\sigma\in\binom{V}{k}$ we have
\[
    \sum_{v\in\sigma}\deg(v)=2|\{e\in E:\, e\subset\sigma\}|+|\{e\in E:\, |e\cap\sigma|=1\}| 
\]
and
\[
\sum_{\substack{v\in\sigma,u\notin \sigma:\\ \{u,v\}\in E}} 1 =|\{e\in E:\, |e\cap\sigma|=1\}|. 
\]
Therefore, we obtain
\begin{align*}
 \maxksum{k}{L(G)}&\leq
  \max_{\sigma\in \binom{V}{k}}\left(2|\{e\in E:\, e\subset \sigma\}|+2|\{e\in E:\, |e\cap \sigma|=1\}|\right)
 \\
   &= 2\cdot \max_{\sigma\in\binom{V}{k}} \left|\left\{ e\in E:\, e\cap \sigma\neq \emptyset\right\}\right|.
\end{align*}
Similarly, by Theorem \ref{thm:merris_gersgorin},
\[
    \maxksum{k}{A(G)} \leq \max_{\sigma\in\binom{V}{k}} \sum_{\substack{v\in\sigma,u\notin \sigma:\\ \{u,v\}\in E}} 1 = \max_{\sigma\in\binom{V}{k}}  \left|\left\{ e\in E:\, |e\cap \sigma|=1 \right\}\right|.
    \]
\end{proof}

%%% AUTHOR: optional appendix here
%\appendix %% you may comment this out if no Appendix
%\section*{Appendix}

%%% AUTHOR: optional acknowledgments here
\section*{Acknowledgments} %%  you may comment this out if no Ackno
I thank the anonymous reviewers for their helpful remarks.

%%% AUTHOR:
%%% Bibliography goes here. Note that the arXiv cannot process bibtex
%%% or biber bibliographies.  Example of acceptable bibliograpy format:
\bibliographystyle{amsplain}

%% AUTHOR: You can generate such a bibliography from a .bib file by 
%% running pdflatex/bibtex/pdflatex/pdflatex and then pasting the .bbl file
%% between \begin{thebibliography} and \end{bibliography}

%%% AUTHOR: Include a short description of each author following the
%%% structure below. Use the same short tags used previously.  
%%% Use \imageat{} and \imagedot{} instead of "@" and "." in
%%% email addresses-this replaces the symbols with graphics to avoid 
%%% e-mail address harvesting from the .pdf file
\begin{dajauthors}
\begin{authorinfo}[pgom]
  Alan Lew\\
  Department of Mathematical Sciences\\
  Carnegie Mellon University\\
  Pittsburgh, PA 15213, USA\\
  alanlew\imageat{}andrew\imagedot{}cmu\imagedot{}edu
\end{authorinfo}
\end{dajauthors}

\end{document}